\documentclass[a4paper]{amsart}
\usepackage{amssymb,amscd,verbatim,latexsym}
\usepackage{enumerate}
\usepackage[all,pdf]{xy}

\usepackage[T1]{fontenc}
\usepackage[american]{babel}
\usepackage[utf8]{inputenc}

\usepackage{lmodern}
\usepackage{microtype}

\usepackage[mathscr]{eucal}

\usepackage{color}

\newtheorem{theorem}{Theorem}[section]

\newtheorem{proposition}[theorem]{Proposition}
\newtheorem{corollary}[theorem]{Corollary}

\newtheorem{lem}[theorem]{Lemma}

\theoremstyle{definition}
\newtheorem{definition}[theorem]{Definition}
\newtheorem{dfn}[theorem]{Definition}
\theoremstyle{remark}
\newtheorem{rem}[theorem]{Remark}
\newtheorem{remark}[theorem]{Remark}
\newtheorem{example}[theorem]{Example}
\newtheorem{exam}[theorem]{Example}

\newcommand{\bbC}{\mathbb C}
\newcommand{\bbK}{\boldsymbol{\mathbb K}}
\newcommand{\bbN}{\mathbb N}
\newcommand{\bbR}{\mathbb R}
\newcommand{\bbS}{\mathbb S}
\newcommand{\bbZ}{\mathbb Z}

\newcommand{\bj}{\mathbf j}
\newcommand{\bs}{\mathbf s}
\newcommand{\bt}{\mathbf t}
\newcommand{\bx}{\mathbf x}
\newcommand{\bu}{\mathbf u}
\newcommand{\bm}{\mathbf m}
\newcommand{\bi}{\mathbf i}

\newcommand{\cC}{\mathcal C}
\newcommand{\cG}{\mathcal G}

\newcommand{\cJ}{\mathcal J}
\newcommand{\cK}{\mathcal K}
\newcommand{\cS}{\mathcal S}

\newcommand{\fU}{\mathfrak U}

\newcommand{\rD}{\mathrm D}
\newcommand{\rE}{\mathrm E}

\newcommand{\tto}{\rightrightarrows}

\DeclareMathOperator{\Forall}{\forall}

\DeclareMathOperator{\ind}{ind}

\DeclareMathOperator{\id}{id}
\DeclareMathOperator{\tr}{tr}

\DeclareMathOperator{\AS}{AS}
\DeclareMathOperator{\Tr}{Tr}
\DeclareMathOperator{\sTr}{sTr}

\begin{document}

\author[Yu Qiao]{Yu Qiao}

\address{School of Mathematics and Statistics, Shaanxi Normal University,
Xi'an, 710119, Shaanxi, China} \email{yqiao@snnu.edu.cn}

\author[Bing Kwan So]{Bing Kwan So}
\address{School of Mathematics, Jilin University,
Changchun, 130023, Jilin, China} \email{bkso@graduate.hku.hk}

\thanks{Qiao was partially supported by the NSFC(11971282),
Natural Science Basic Research Program of Shaanxi (Program No. 2020JM-278),
and the Fundamental Research Funds for the Central Universities (GK202003006).}

\date{\today}
\title[Renormalized Index Formulas on Boundary Groupoids]{Renormalized Index Formulas for Elliptic Differential Operators on Boundary Groupoids}

\maketitle

\begin{abstract}
We consider the index problem of certain boundary groupoids of the form
$\cG = M _0 \times M _0 \cup \mathbb{R}^q \times M _1 \times M _1$.
Since it has been shown that for the case that $q \geq 3$ is odd, $K _0 (C^* (\cG)) \cong \bbZ $,
and moreover the $K$-theoretic index coincides with the Fredholm index,
we attempt in this paper to derive a numerical formula for elliptic differential operators on renormalizable boundary groupoids.
Our approach is similar to that of renormalized trace of Moroianu and Nistor \cite{Nistor;Hom2}.
However, we find that when $q \geq 3$, the eta term vanishes,
and hence the $K$-theoretic and Fredholm indices of elliptic (respectively fully elliptic)
pseudo-differential operators on these groupoids are given only by the Atiyah-Singer term.
As for the $q=1$ case we find that the result depends on how the singularity set $M_1$ lies in $M$.

\end{abstract}

\tableofcontents

\section{Introduction}\label{intro}

The Atiyah-Singer index theorem is one of the most celebrated results of the twentieth century mathematics.
It states that given a compact manifold $M$:
\begin{enumerate}
\item[(i)]
Given a closed manifold $M$, any elliptic differential operator $D$ acting between sections of vector bundles over $M$ is invertible modulo compact operators,
hence $D$ is Fredholm;
\item[(ii)]
The Fredholm index of $D (D D^* + \id )^{-\frac {1}{2}}$ is its image under the connecting map in $K$-theory induced by the short exact sequence
$$ 0 \to \cK \to \fU \to \fU / \cK \to 0 ;$$
\item[(iii)]
One has $K _0 (\cK) \cong \bbZ $, and moreover,
under such isomorphism, an explicit formula for the index, which produces an integer, is given:
$$\partial ([D (D D^* + \id )^{-\frac {1}{2}}]) = \int _M \AS (D) .$$
\end{enumerate}

There has been many attempts to generalize the Atiyah-Singer index theorem.
One major direction is suggested by Connes in \cite[Section 2]{Connes;Book}.
Here one considers any Lie groupoid $\cG \rightrightarrows M$ with $M$ compact
(the classical case corresponds to the pair groupoid $\cG = M \times M$).
On $\cG$ one constructs the pseudo-differential calculus
\cite{Rouse;BlupInd,Nistor;LieMfld,Bohlen;Rev1,Debord;FiberedCorners,Debord;Blup,M'bert;CornerGroupoids,NWX;GroupoidPdO}.
The sub-algebras of operators of order zero and $-\infty$ can be completed to $C^*$-algebras
$\fU (\cG)$ and $C^* (\cG)$, respectively.
Then one shows that any elliptic operator is invertible modulo order $-\infty $ order operators.
It follows that for any elliptic operator of order zero, one can consider its analytic index in $K _0 (C^* (\cG))$ through the short exact sequence
$$ 0 \to C^* (\cG) \to \fU (\cG) \to \fU / C^* (\cG) \to 0.$$
Equivalently, the analytic index can be defined by Lauter, Monthubert, and Nistor in \cite{Nistor;Family}, or by Connes via the adiabatic groupoid \cite[Chapter 2]{Connes;Book}.
The objective of the index theorem is to give an alternative description of such index, preferably in ``topological'' terms.

Given a regular foliation $\mathcal F \subset T M$,
one considers the holonomy groupoid $\cG$.
Generalizing the Atiyah-Singer index theorem,
one considers an embedding $M \to \bbR ^N$ for some sufficiently large $N$.
Then one can define a topological index using some Thom isomorphism and Morita equivalence (see for instance \cite[Chapter 2.9]{Connes;Book}).

The index theory for groupoids with leaves of different dimensions is less known.
The most well studied cases are mostly related to manifolds with corners.
The $K$-theory of the interior of the $b$-stretched product $\cG \rightrightarrows M$ is computed
by Carrillo Rouse and Lescure in \cite{Rouse;CornerFredholm}, and Le Gall and Monthubert in \cite{Monthubert;BdK}.
Recently  Yamashita in \cite{Yamashita;FiberedBdK} computed the fibered boundary case corresponding to the $\Phi$ and edge calculus.
In the particular case of manifolds with boundary, the $K$-groups are trivial.
Index theorems are again proved in \cite{Rouse;BlupInd,Monthubert;BdCoh,Monthubert;BoundaryInd,Nistor;ConicalInd,Debord;FiberedCorners,Nistor;TopIndCorn}.
These results rely heavily on the boundary defining function
(in other words the groupoid $\cG$ under consideration must be the result of some blowups using the boundary defining function).
For example, Akrour and Carrillo Rouse \cite{Rouse;BlupInd}, and Monthubert and Nistor \cite{Nistor;TopIndCorn}, used an auxiliary embedding of $M$ into a cube,
and the fact that given any transverse embedding of manifold with corners $M_1 \to M _2$,
one can construct explicitly, using the boundary defining functions,
a groupoid homomorphism from the $b$-stretched product of $M _1$ to $b$-stretched product of $M _2$;
while in the works of Carrillo Rouse, Lescure, and Monthubert \cite{Monthubert;BdCoh},
Debord, Lescure, and Nistor \cite{Nistor;ConicalInd}, and Debord, Lescure, and Rochon \cite{Debord;FiberedCorners}, various groupoids are constructed in place of the tangent Lie algebroid.

On the other hand, one may also consider the Fredholm index of a Fredholm operator,
that is, an elliptic operator invertible modulo $C^* (M_0 \times M_0) \cong \cK$,
the $C^*$-algebra of compact operators.
Here, recall \cite{Nistor;GeomOp} that an elliptic operator $D$ is Fredholm if $D $ is invertible over all singular leaves
(see \cite{Nistor;Fredholm1,Nistor;Fredholm2} for a discussion of the converse).
By using renormalized traces, or otherwise,
one proves the Atiyah-Patodi-Singer index theorem in the case of manifolds with boundary \cite{Albin;EdgeInd,Loya;HeatIndex,Melrose;Book,Nistor;Hom2},
which is of the form
$$\partial ([D (D D^* + \id )^{-\frac {1}{2}}]) = \int _M \AS (D) + \eta (D),$$
where $\eta$ is an explicitly given non-local term, that is, in general, not determined by the principal symbol of $D$.
One also gets results of a similar form for some Lie manifolds \cite{Bohlen;HeatIndex,Nistor;Hom2} and the Bruhat sphere \cite{So;PhD}.
However, in many cases the explicit form of the $\eta$ term is unknown.

The notion of boundary groupoids is introduced in \cite{So;FullCal}, as a generalization of manifolds with boundary.
They also arise as the holonomy groupoid integrating some singular foliations \cite{Nistor;Riem,Nistor;LieMfld,Skandalis;SingFoliation1,Skandalis;SingFoliation2,Debord;IntAlgebroid},
or symplectic groupoid of some Poisson manifolds \cite{Lu;PoissonCohNotes}.
Roughly speaking, a boundary groupoid is a Lie groupoid of the form:
$$\cG = \bigcup _i G _i \times M _i \times M _i .$$
Here, it is important to note that the expression above does {\it not} uniquely define a Lie groupoid.
For example, the $b$-calculus, cusp calculus, and the very small calculus all integrate to
$$\cG = M _0 \times M _0 \cup \bbR \times M _1 \times M _1, $$
but each has different smooth structures and Lie algebroids.
For the integrability of Lie algebroids, one may consult the works of Crainic and Fernandes \cite{Fern'd;IntAlgebroid}, Debord \cite{Debord;IntAlgebroid}, and Nistor \cite{Nistor;IntAlg'oid}.
In \cite{So;FullCal}, the second author constructed a calculus of pseudo-differential operators, using only intrinsic quantities like Riemannian distance functions.
The calculus includes the finite rank parametrix of Fredholm operators,
and could be served as a framework for various problems (for example \cite{Nistor;Polyhedral3}).
Thus, it appears that these groupoids are representatives when one considers groupoids with leaves of different dimensions.


In \cite{SoKtheory}, Carrillo Rouse and the second author computed the $K$-groups of some boundary groupoids,
using composition series exact sequence \cite{Nistor;Family}.
Explicitly, for a boundary groupoid of the form
$$\cG = M _0 \times M _0 \cup G \times M _1 \times M _1,$$
where $G$ is exponential (i.e., a connected, simply-connected and solvable Lie group),
one has
\begin{align*}
K _0 (C^*(\cG)) \cong \bbZ, & \quad K _1 (C^*(\cG)) \cong \bbZ, & \text{if $\dim G \geq 3$ odd} ;\\
K _0 (C^*(\cG)) \cong \bbZ \oplus \bbZ , & \quad K _1 (C^*(\cG)) \cong \{0\}, & \text{if $\dim G$ even.}
\end{align*}
Moreover, when $\dim G$ is odd, one has the commutative diagram
\begin{equation}
\begin{gathered}
\xymatrix{
& K_1 (C^* (\bar \cG _1)) \ar[r]  &  K _1 (\fU / C^ *(\cG _0 ))\ar[r] \ar[d]_{\partial} & K _1 (\fU / C^* (\cG)) \ar[r] \ar[d]^{\partial} & K _0 (C^* (\bar \cG_1)) \\
& &  K _0 ( C^ *(\cG _0 ))\ar[r]^\cong                 & K _0 (C^* (\cG)),
}
\end{gathered}
\end{equation}
where the left column is just the Fredholm index map.
It follows that for a fully elliptic operator, its index in $K_0 (C^* (\cG))$ equals to its Fredholm index.

In this paper, we turn to the problem of the computation of the index.
Suppose we are given a fully elliptic operator $D$.
Let $\varPhi$ be an inverse of $D$ up to finite rank operators in $C^* (\cG _0)) \cong \cK$, then its Fredholm index is given by
$$ \ind (D)
= \int _{M_0} \tr ((D \varPhi - 1) |_{M_0}) \mu _{M_0} - \int _{M_0} \tr ((\varPhi D - 1) |_{M_0}) \mu _{M_0},$$
where the Riemannian metric on $M _0$ is defined by fixing a metric on $A \to M$, restricting to $M _0$,
and then identifying $A |_{M _0} $ with $T M_0$ via the inverse of the anchor map,
therefore the corresponding Riemannain volume form $\mu_{M _0}$ has infinite volume.
In order to apply the Calderon formula, we suppose further that $\cG$ is of a special form.
Namely, to use cylindrical coordinates, we let $r$ be some Riemannian distance from $M_1$,
and we further suppose there is a system of coordinates around $M_1$ such that the structural vector fields are given by
$$
r ^N\frac{\partial }{\partial x _1} , \cdots, r^N \frac{\partial }{\partial x _q},
\frac {\partial }{\partial y_1} , \cdots, \frac {\partial }{\partial y_{\dim M _1}}.
$$
This implies the Riemannian density of the induced metric on $M _0$ is of the form
$$
\mu _{M_0} = r^{-N q+(q-1)}\,d r \,d\theta \, d y
$$
(see Example \ref{RenormEx}).
Then we are essentially in the same situation as the cusp metric \cite{Nistor;Hom2} and we extend the trace formula to a renormalized trace
$$ \widehat \Tr ([D, \varPhi])
:= \frac{\partial ^2}{\partial z \partial \tau} \Big|_{\tau = z = 0} \int _{M _0} \tau z \tr (r ^z [D, \varPhi ] Q ^{- \tau} |_{M_0})\mu _{M_0} $$
(where the variables $\tau, z$ are complex).
By a similar argument of Nistor and Moroianu \cite{Nistor;Hom2},
the renormalized trace above can be written as the sum of an Atiyah-Singer term and eta term.
However, unlike \cite{Nistor;Hom2}, we find that when the codimension is odd and greater than or equal to 3,
the eta term vanishes, leaving the Atiyah-Singer term (see Theorem \ref{RenormIndex}).
\begin{theorem}
\label{Main}
For any elliptic, differential operator $D \in \rD ^1 (\cG , E)$,
the index of $D$ in $K _0 (C^* (\cG))$ equals
\begin{equation}
\partial ([D (D D^* + \id )^{- \frac{1}{2}} ])
= \int_{M _0} r ^{z - q N} \, a _0 \, \mu_M \Big |_{z = 0},
\end{equation}
where $a _0 $ is the constant term in the $t \to 0 $ asymptotic expansion of the super-trace of heat kernel.
\end{theorem}


\subsection{Contents of the paper}

Section 2 recalls the background of necessary terminology we shall use in this paper,
such as Lie groupoids, Lie algebroids, pseudodifferential calculus on Lie groupoids, boundary groupoids and
their composition series.
Throughout the paper, we shall consider the simplest form of boundary groupoids
$$\cG = M _0 \times M _0 \cup \bbR ^q  \times M _1 \times M _1 ,$$
for odd positive integer $q$, and explore index formulas for the cases $q\geqslant 3$ odd and $q=1$
in Sections \ref{OddRenorm} and \ref{qoneCase}, respectively.

In Section \ref{OddRenorm}, we assume that $q$ is an odd positive integer greater than $3$. By the work of Carrillo Rouse and the
second author \cite{SoKtheory}, the Fredholm index is in this case equivalent to the $K$-theoretic index.
With this result in mind, we turn to consider the index of elliptic operators.
We assume that $\dim M_1 \geqslant 1$ and $q \geqslant 3$ odd, i.e., the singular structure is similar to the very small calculus,
a renormalized index along the lines of \cite{Bohlen;HeatIndex,Nistor;Hom2} can be defined,
and we shall show in addition that the $\eta$ term vanishes, leaving only the Atiyah-Singer term.

In Section \ref{qoneCase}, we assume that $q=1$. More precisely, in Subsection \ref{BoundaryCase}, we consider the case that
$M_1$ is the connected boundary of $M:=M_0 \cup M_1$. Much about this case is known, and we
reformulate an index formula in the context of boundary groupoids.
In Subsection \ref{TwoHyperSurfaces}, we consider the case that $M$ is a connected manifold without boundary partitioned by a co-dimension 1 submanifold $M _1$,
i.e., $M _0 \setminus M _1$ has two components. The corresponding Lie groupoid is of the form
$$\cG = (M_0' \times M_0') \cup (M_0'' \times M_0'') \cup (\mathbb{R} \times M_1 \times M_1).$$
(Strictly speaking, $\cG$ is {\em not} a boundary groupoid.)
We show that $K _0 (C^* (\cG)) \cong \bbZ$, hence it is meaningful to consider the index in $K _0 (C^* (\cG))$.
We prove an index formula for this case in Theorem \ref{OneInd},

In Section \ref{summary}, we point out some future research directions.

\smallskip
\textbf{Acknowledgments}
The authors are grateful to Victor Nistor for helpful comments, enlightening discussions and email communications.



\bigskip
\section{Preliminaries}

\subsection{Lie groupoids, Lie algebroids, and groupoid $C^*$-algebras}

We review first basic knowledge of Lie groupoids, Lie algebroids, pseudodifferential calculus
on Lie groupoids, and groupoid $C^*$-algebras \cite{Nistor;LieMfld, Nistor;Family, Nistor;GeomOp, MacBook87, MacBook05, MonPie, NWX;GroupoidPdO, RenBook80, So;FullCal}.

\begin{definition}
A Lie groupoid $\cG \tto M$ consists of the following data:
\begin{enumerate}
\item[(i)]
 manifolds $M$, called the space of units, and $\cG$;
\item[(ii)]
 a unit map (inclusion) $\bu: M \rightarrow G$;
\item[(iii)]
submersions $\bs,\bt: \cG \rightarrow M$, called the source and target maps respectively, satisfying
$$\bs \circ \bu= \id_M =\bt\circ \bu ;$$
\item[(iv)]
 a multiplication map $\bm: \cG^{(2)}:=\{(g, h)\in \cG \times \cG \, : \, \bs(g) = \bt(h)\} \rightarrow \cG$, $(g,h) \mapsto gh$, which
is associative and satisfies
$$\bs(gh)=\bs(h), \quad \bt(gh)=\bt(g), \quad g(\bu \circ \bs (g))=g =(\bu\circ \bt(g))g;$$
\item[(v)]
an inverse diffeomorphism $\bi: \cG \rightarrow \cG$, $g \mapsto g=g^{-1}$, such that $\bs(g^{-1})=\bt(g)$, $\bt(g^{-1})= \bs(g)$, and
$$g g^{-1}=\bu(\bt(g)), \quad g^{-1}g= \bu(\bs(g)).$$
\end{enumerate}
\end{definition}

\begin{remark}
In general, the space $\cG_1$ is not required to be Hausdorff. However, since $\bs$ is a submersion,
it follows that each fiber $\cG_x:=\bs^{-1}(x)$ is a smooth manifold, hence it is Hausdorff.
In this paper, all groupoids are assumed to be Hausdorff.
\end{remark}

We sometimes use the following notations. Let $A, B$ be subsets of $M$. we denote by $\cG_A:=\bs^{-1}(A)$, $\cG^A:=\bt^{-1}(A)$, and $\cG_A^B:=\cG_A \cap \cG^B$.
We call $\cG_A^A$ the {\em reduction} of $\cG$ to $A$. In particular, if $A=\{x\}$, then we write $\cG_x=\cG_A$ and $\cG^x=\cG^A$ respectively.
For any $x\in M$, $\cG_x^x=\bs^{-1}(x) \cap \bt^{-1}(x)$ is a group, called the {\em isotropy group} at $x$.

To introduce the ``Lie algebroid'' of a Lie groupoid, let us recall the definition of Lie algebroids.

\begin{definition}
A \emph{Lie algebroid} $A$ over a manifold $M$ is a vector bundle $A$
over $M$, together with a Lie algebra structure on the space
$\Gamma(A)$ of the smooth sections of $A$ and a bundle map $\nu:
A\rightarrow TM$, extended to a map between sections of theses
bundles, such that
\begin{enumerate}
\item[(i)] $\nu([X,Y])=[\nu(X),\nu(Y)]$;
\item[(ii)] $[X,fY]=f[X,Y]+(\nu(X)f)Y$,
\end{enumerate}
for all smooth sections $X$ and $Y$ of $A$ and any smooth function $f$
on $M$. The map $\nu$ is called the \emph{anchor}.
\end{definition}

Give a Lie groupoid $\mathcal{G}$ with units $M$, we can associate
a Lie algebroid $A(\cG)$ to $\cG$ as follows. (For more details, one can read \cite{MacBook87, MacBook05}.) The
$\bs$-vertical subbundle of $T\mathcal{G}$ for $\bs:\mathcal{G}\rightarrow
M$ is denoted by $T^\bs(\mathcal{G})$ and called simply the $\bs$-vertical
bundle for $\mathcal{G}$. It is an involutive distribution on
$\mathcal{G}$ whose leaves are the components of the $\bs$-fibers of
$\mathcal{G}$. (Here involutive distribution means that
$T^\bs(\mathcal{G})$ is closed under the Lie bracket, i.e. if $X,Y \in
\mathfrak{X}(\mathcal{G})$ are sections of $T^\bs(\mathcal{G})$, then
the vector field $[X,Y]$ is also a section of $T^\bs(\mathcal{G})$.) Hence
we obtain
\begin{equation*}
  T^\bs\mathcal{G}=\text{ker} \ \bs_*=\displaystyle{\bigcup_{x\in
      M}T\mathcal{G}_x}\subset T\mathcal{G}.
\end{equation*}
The \emph{Lie algebroid} of $\cG$, denoted by $A(\mathcal{G})$ (or simply $A$ sometimes), is
defined to be $T^\bs(\mathcal{G})|_M$, the restriction of the
$\bs$-vertical tangent bundle to the set of units $M$. In this case, we
say that $\cG$ {\em integrates} $A(\cG)$.


Let $E \to M$ be a vector bundle.
Recall \cite{NWX;GroupoidPdO} that an $m$th order pseudo-differential operator on $\cG$ is a right invariant,
smooth family $P = \{ P _x \} _{x \in M}$,
where each $P _x $ is an $m$th order classical pseudo-differential operator on sections of $\bt ^* \rE \to \bs ^{-1} (x)$.
We denote by $\Psi ^m (\cG, E)$ (resp. $\rD ^m (\cG , E)$) the algebra of uniformly supported,
order $m$ classical pseudo-differential operators (resp. differential operators).

Suppose that $\cG \tto M$ is a Lie groupoid with $M$ compact.
Then there is an unique (up to equivalence) Riemannian metric on its Lie algebroid $A(\cG) \to M$,
which induces a family of right invariant Riemannian metrics on $\bs ^{-1} (x), x \in M$.
We shall denote by $\mu _x$ the Riemannian volume forms (which is a Haar system on $\cG$), and $d (\cdot, \cdot)$, the Riemannian distance function.

Recall \cite{Nistor;Family} that one defines the strong norm for $P \in \Psi^0(\cG, E)$
$$ \| P \| := \sup _\rho \| \rho (P) \|, $$
where $\rho $ ranges over all bounded $*$-representations of $\Psi ^0 (\cG, E)$ satisfying
$$ \| \rho (P) \| \leq \sup _{x \in M}
\Big\{ \int _{\bs ^{-1} (x)} | \kappa _P (g) |\, \mu _x , \int _{\bs ^{-1} (x)} | \kappa _P (g ^{-1}) | \, \mu _x \Big\},$$
whenever $P \in \Psi ^{- \dim M - 1} (\cG, E)$ with (continuous) kernel $\kappa _P$.

\begin{dfn}
The $C^*$-algebras $\fU (\cG)$ and $C ^* (\cG)$ are defined to be the completion of $\Psi ^0 (\cG, E)$ and $\Psi ^{- \infty} (\cG, E)$ respectively
with respective to the strong norm $\| \cdot \| $.
\end{dfn}

One also defines the reduced $C^*$-algebras $\fU _r (\cG)$ and $C^* _r (\cG)$ by completing
$\Psi ^0 (\cG, E)$ and $\Psi ^{- \infty} (\cG, E)$) respectively with respect to the reduced norm
$$ \| P \| _r := \sup _{x \in M} \big\{ \| P _x \| _{L^2 (\bs ^{-1} (x))} \big\}.$$
Recall that the strong and reduced norm coincides if $\cG$ is (metrically) amenable,
which is the case for the groupoids we shall consider.

Suppose further that $\cG$ has polynomial volume growth \cite{So;FullCal}.
Then convolution product is well defined between kernels with faster than polynomial decay, and hence one may define:
\begin{dfn}
The sub-algebra of Schwartz kernels $S^* (\cG) \subset C ^* _r (\cG)$ is the set of all kernels $\psi \in C^*(\cG)$ satisfying:
\begin{enumerate}
\item[(i)]
$\psi $ is continuous;
\item[(ii)]
$\psi |_{\cG _i}$ is smooth for all $i$;
\item[(iii)]
For any collection of sections $V_1, \cdots, V_k \in \Gamma ^\infty (A)$,
regarded as first order differential operators (right invariant vector fields) on $\cG$ and $N =0, 1, 2, \cdots$,
$$ a \mapsto (V_1 \cdots V_i \psi V _{i+1} \cdots V _k )(d (a, \bs (a)) + 1)^{-N}$$
is a bounded function.
\end{enumerate}
\end{dfn}

The significance of $S ^* (\cG)$ lies in (for instance \cite{Bohlen;HeatIndex}, \cite[Section 6]{Nistor;Funct}):
\begin{lem}
\label{FunctCal}
The sub-algebras $S ^* (\cG) \subset C^* _r (\cG)$ and $\Psi ^0 (\cG , E) \oplus S ^* (\cG) \subset \fU _r (\cG)$ are dense,
and are invariant under holomorphic functional calculus.
\end{lem}

\smallskip
\subsection{Invariant submanifolds and composition series}

Let $\cG \tto M$ be a Lie groupoid.
\begin{dfn}
We say that an embedded sub-manifold $M' \subset M$ is an {\em invariant submanifold} if
$$ \bs^{-1} (M') = \bt ^{-1} (M' ) .$$
\end{dfn}

For any invariant submanifold $M '$,
we denote the sub-groupoid $ \bs^{-1} (M ') = \bt ^{-1} (M ') $ by $\cG _{M '}$ as before.

\begin{dfn}
Given a closed invariant submanifold $M' $ of $\cG$.
For any $P = \{ P _x \} _{x \in M} \in \Psi ^m (\cG, E)$, we define the restriction
$$ P |_{\cG _{M'}} := \{ P _x \} _{x \in \cG _{M'}} \in \Psi ^m (\cG _{M'}, E).$$
Restriction extends to a map from $\fU (\cG) $ to $\fU (\cG _{M'})$ and also from $C^* (\cG)$ to $C^* (\cG _{M'})$.
\end{dfn}

Now suppose we are given a groupoid $\cG \rightrightarrows M$ ($M $ not necessarily compact),
and closed invariant submanifolds
$$ M = \bar{M _0} \supset \bar{M _1} \supset \cdots \supset \bar{ M _r }.$$
For simplicity we shall denote $\bar \cG _i := \cG _{\bar{M _i}}$.
Recall that $S A'$ denotes the sphere sub-bundle of the dual of the Lie algebroid $A(\cG)$ of $\cG$.

\begin{dfn}
Let $\sigma : \Psi ^m (\cG) \to C^\infty (S A')$ denotes the principal symbol map.
For each $i = 1, \cdots , r$, define the {\em joint symbol maps}
\begin{equation}
\label{JointSym}
\bj _i : \Psi ^m (\cG) \to C^\infty (S A') \oplus \Psi ^m (\bar \cG _i),
\quad \bj _i (P ) := (\sigma (P ) , P |_{\bar \cG _i}).
\end{equation}
The map $\bj _i$ extends to a homomorphism from $\fU (\cG)$ to $C_0 (S A ^*) \oplus \fU (\bar \cG _i)$.

We say that $P \in \Psi ^m (\cG)$ is {\em elliptic} if $\sigma (P)$ is invertible,
and it is {\em fully elliptic} if $\bj _1 (P )$ is invertible.
\end{dfn}

\begin{dfn}
Denote by $\cJ _0 := \overline {\Psi ^{-1} (\cG)} \subset \fU (\cG)$,
and $\cJ _i \subset \fU (\cG), i = 1, \cdots , r$ the null space of $\bj _{r - i + 1}$.
\end{dfn}

By construction, it is clear that
$$ \cJ _0 \supset \cJ _1 \supset \cdots \supset \cJ _r .$$
Also, any uniformly supported kernels in $\Psi ^{- \infty} (\cG _{\bar M _i \setminus \bar M _j}, E |_{\bar M _i \setminus \bar M _j})$,
can be extended to a kernel in $\Psi ^{- \infty }  (\cG _{\bar M _i}, E |_{\bar M _i })$ by zero.
This induces a $*$-algebra homomorphism from $C^* (\cG _{\bar M _i \setminus \bar M _j} ) $ to $ C^* (\cG _{\bar M _i}) $.
We shall use the following key fact.

\begin{lem}\cite[Lemma 2 and Theorem 3]{Nistor;Family}
\label{Comp}
One has short exact sequences
\begin{align}
& 0 \to \cJ _{i+1} \to \cJ _i \to C ^* (\cG _{\bar M _i \setminus \bar M _{i +1}}) \to 0 , \\
& 0 \to C^* (\cG _{\bar M _i \setminus \bar M _j} ) \to C^* (\bar \cG _i) \to C^* (\bar \cG _j) \to 0, \quad \Forall j > i.
\end{align}
\end{lem}

\medskip
\subsection{Boundary groupoids and their composition series}

\begin{dfn}
\label{Dfn}
Let $\cG \rightrightarrows M$ be a Lie groupoid with $M$ compact.
We say that $\cG$ is a boundary groupoid if:
\begin{enumerate}
\item[(i)]
the singular foliation defined by the anchor map $\nu : A \to TM $ has finite number of leaves
$M _0 , M _1, \cdots , M _r \subset M$
(which are invariant submanifolds),
such that $\dim M = \dim M _0 > \dim M _1 > \cdots > \dim M _r $;
\item[(ii)]
For all $k = 0, 1, \cdots , r$, $\bar M _k := M _k \cup \cdots \cup M _r $ are closed submanifolds of $M$;
\item[(iii)]
$\cG _0 := \cG _{M _0}$ is the pair groupoid,
and $\cG _k := \cG _{M _k} \cong G _k \times M _k \times M _k $ for some Lie group $G _k$;
\item[(iv)]
For each $k$, there exists an unique sub-bundle $\bar A _k \subset A |_{\bar M _k}$ such that
$\bar A _k |_{M _k} = \ker (\nu |_{M _k}) $ ($= \mathfrak g _k \times M _k$).
\end{enumerate}
\end{dfn}

Boundary groupoids are closely related to Fredholm groupoids and blowup groupoids. Here we recall some basic properties of boundary groupoids.

By the same arguments as in \cite{Nistor;GeomOp}, we can prove the following lemma.

\begin{lem}
(See \cite[Lemma 7]{Nistor;GeomOp})
For any boundary groupoid of the form $\cG = (M _0 \times M _0) \cup (\bbR ^q \times M _1 \times M _1)$,
$$ C^* (\cG) \cong C^* _r (\cG).$$
In other words, $\cG$ is (metrically) amenable.
\end{lem}

Moreover, since the additive group $\bbR^q$ is amenable, the groupoid of the form $\bbR ^q \times M _1 \times M _1$ is (topologically) amenable.
By \cite[Theorem 4.3]{Come19}, we have the following proposition.

\begin{proposition}
The groupoid $\cG = (M _0 \times M _0) \cup (\bbR ^q \times M _1 \times M _1)$ is a Fredholm groupoid.
\end{proposition}

\begin{remark}
For the definition and basic properties of Fredholm groupoids, one may consult \cite{Nistor;Fredholm1, Nistor;Fredholm2,Come19}.
Roughly speaking, Fredholm groupoids are those on which Fredholmness of a pseudodifferential operator is completely characterized by its
ellipticity and invertibility at the boundary.
\end{remark}

We are in position to give some examples of boundary groupoids.

\begin{example}
Let $\bbR$ act on $[-\infty , \infty]$ by addition, with $g + \infty := \infty, g + (- \infty ) := - \infty $ for any $g \in \bbR$.
Then the action groupoid $[ -\infty, \infty] \rtimes \bbR $ is a boundary groupoid.
\end{example}

\begin{example}\cite[Section 2.1]{SoKtheory}
Fix a smooth function $f:\mathbb{R}^q\to \mathbb{R}$ with $f(0)=0$ and $f > 0$ over $\mathbb{R}^q\setminus \{0\}$.
Then there exists a Lie algebroid structure over
$$A_f:=T\mathbb{R}^q\to \mathbb{R}^q$$
with anchor map
$$(z,w)\mapsto (z,f(z)\cdot w) , \quad \Forall (z,w)\in T\mathbb{R}^q=\bbR^q\times \bbR^q$$
and Lie bracket on sections
$$ [X, Y] _{A _f} := f [X, Y] + (X \cdot f) Y - (Y \cdot f ) X ,$$
whose image under the anchor map defines the foliation.
This Lie algebroid is almost injective
(since the anchor is injective, in fact an isomorphism in this case, over the open dense subset $\mathbb{R}^q\setminus \{0\}$),
and hence it is integrable by a Lie groupoid that is a quasi-graphoid (see \cite[Theorems 2 and 3]{Debord;IntAlgebroid}),
which is of the form
$$ (\bbR ^q \setminus \{0 \}) \times (\bbR ^q \setminus \{ 0 \}) \cup \bbR ^q.$$
\end{example}

\begin{exam}[Renormalizable boundary groupoid]
\label{RenormEx}
Now let $M$ be a compact manifold of dimension $n$ without boundary, $M_1 \subset M$ be a closed submanifold of codimension $q \geq 2$.
Fix a Riemannian metric $g_M$ on $M$. Let $r : M \to \bbR$ be the Riemannian distance from $M_1$. Then $r^2$ is smooth near $M_1$.
Fix an even number $N \geq 2$.
Take coordinate charts $(x _\alpha , y _\alpha )$ near $M_1$ such that
$$M_1 \cap U _\alpha = \{ x _\alpha ^1 = \cdots = x _\alpha ^{q} = 0 \}$$
and
$$ r ^2 = (x_\alpha ^1 )^2 + \cdots + (x_\alpha ^q )^2 .$$
We consider the singular foliation spanned by
\begin{equation}
\label{StructuralVF}
r ^N\frac{\partial }{\partial x _1} , \cdots, r^N \frac{\partial }{\partial x _q},
\frac {\partial }{\partial y_1} , \cdots, \frac {\partial }{\partial y_{\dim M _1}}
\end{equation}
on $U _\alpha $.
These vector fields are locally free, and closed under Lie bracket,
hence by \cite{Nistor;Riem} they are the image of some Lie algebroid $A$.
Indeed, $A$ can be explicitly constructed similar to the last example.

In turn, $A$ can be integrated into a Lie groupoid $\cG$.
Locally, over $U _\alpha$, $\cG _{U _\alpha}$ is of the form
$$ \cG _{U _\alpha } = (U'_\alpha \times U' _\alpha ) \times \big((U''_\alpha \times U''_\alpha) \cup \bbR^q) \big),$$
where $U'' _\alpha \subset \bbR^q $ and $(U''_\alpha \times U''_\alpha) \cup \bbR^q $ are defined in the previous example.
Then using \cite[Section 2.2]{SoKtheory} or \cite{Nistor;IntAlg'oid} one can ``glue'' these groupoids together to form a groupoid over $M$.
\end{exam}


Given a boundary groupoid, we consider the sequence of invariant submanifolds,
$$ M \supset \bar M _1 \supset \cdots \supset \bar M _r, $$
where $\bar M _i$ is given to be (2) of Definition \ref{Dfn}.
The second short exact sequence of Lemma \ref{Comp} then induces the six-terms exact sequences:
\begin{equation}\label{Comp2}
\begin{CD}
K _1 (C^* (\cG _{M _i} )) @>>> K _1 (C^* (\bar \cG _i)) @>>> K _1(C^* (\bar \cG _{i+1})) \\
@AAA @. @VVV \\
K _0(C^* (\bar \cG _{i+1})) @<<< K _0 (C^* (\bar \cG _i)) @<<< K _0 (C^* (\cG _{M _i }))
\end{CD}
\end{equation}
If $r = 1$, and  $G _1$ is solvable, connected and simply connected (i.e. exponential), then the system \eqref{Comp2} will be greatly simplified.
In this case, we only have one six-terms exact sequence,
and we have
\begin{align*}
C^* (\cG _{M _0}) = C^* (M _0 \times M _0 ) \cong & C^* _r (M _0 \times M _0 ) \cong \cK ,\\
C^* (\bar \cG _0) = & C^* (\cG),  \\
C^* (\bar \cG _1) = C^* (\cG _1) = & C^* (G _1 \times M _1 \times M _1) \cong \cK \otimes C^* (G _1),
\end{align*}
where the last isomorphism follows from \cite[Proposition 2]{Nistor;GeomOp}.
The $K$-groups of $\cK$ are well known --- $K _0 (\cK) \cong \bbZ, K _0 (\cK) \cong \{0 \}$.
Hence, one uses Connes' Thom isomorphism \cite{Connes;ThomIsomorphism} to compute the $K$ groups of $\cK \otimes C^* (G _1)$.
Then system \eqref{Comp2} reads
\begin{align}
\label{Comp3}
\begin{CD}
\bbZ @>>> K _0 (C^* (\cG)) @>>> \bbZ \\
@AAA @. @VVV \\
0 @<<< K _1 (C^* (\cG)) @<<< 0
\end{CD}
& \quad \text{ if $\dim G_1$ is even,} \\
\label{Comp4}
\begin{CD}
\bbZ @>>> K _0 (C^* (\cG)) @>>> 0 \\
@AAA @. @VVV \\
\bbZ @<<< K _1 (C^* (\cG)) @<<< 0
\end{CD}
& \quad \text{ if $\dim G_1$ is odd.}
\end{align}
In \cite{SoKtheory}, Carrillo Rouse and the second author compute the $K$-theory of boundary groupoids with $r=1$, namely:
\begin{align*}
K _0 (C^*(\cG)) \cong \bbZ, & \quad K _1 (C^*(\cG)) \cong \bbZ, & \text{if $\dim G \geq 3$ odd}; \\
K _0 (C^*(\cG)) \cong \bbZ \oplus \bbZ, & \quad K _1 (C^*(\cG)) \cong \{0\}, & \text{if $\dim G$ even.}
\end{align*}


\smallskip
\subsection{The relation between Fredholm and $K _0 (C ^* (\cG))$ indices}\label{Relation}

Suppose that $\cG \tto M $ is a boundary groupoid in the sense of Definition \ref{Dfn}.
We turn to study the relationship between Fredholm and $K _0 (C ^* (\cG))$ index.

Recall \cite{Nistor;GeomOp} that an elliptic operator $P$ is Fredholm if $P |_{\bar \cG _1}$ is invertible,
and its Fredholm index is its image under the connecting map induced by the short exact sequence
$$ 0 \to C^* (M _0 \times M _0 ) \cong \cK \to \fU \to \fU / \cK \to 0 .$$
Consider the morphism of short exact sequence
\begin{equation}
\label{UMorph}
\begin{CD}
0 @>>> C^* (\cG _0 ) @>>> \fU @>>> \fU /  C^* (\cG _0 ) @>>> 0 \\
@. @VVV @| @VVV @. \\
0 @>>> C^* (\cG) @>>> \fU @>>> \fU / C^* (\cG) @>>> 0,
\end{CD}
\end{equation}
where the left column is the incursion map of Lemma \ref{Comp}.
The right column is clearly defined by quotient out a bigger ideal, hence fits into the short exact sequence
$$ 0 \to C^* (\cG) / C^* (\cG _0) \cong C^* (\bar \cG _1) \to \fU /  C^* (\cG _0 ) \to \fU / C^* (\cG) \to 0.$$
By naturality of the index, one gets:
\begin{equation}\label{Surj}
\begin{gathered}
\xymatrix{
&  K _1 (C^* (\bar \cG _1))\ar[r]  &  K _1 (\fU / C^ *(\cG _0 ))\ar[r] \ar[d]_{\partial} & K _1 (\fU / C^* (\cG))\ar[d]^{\partial} \ar[r] & K _0 (C^* (\bar \cG _1)) \\
&  &   K _0 ( C^ *(\cG _0 ))\ar[r]                 & K _0 (C^* (\cG)).         &
}
\end{gathered}
\end{equation}
Thus, a class $[P] \in K _1 (\fU / C^* (\cG))$ is represented by a Fredholm operator $[\tilde P]$,
if and only if, its index in $K _0 (C^* (\bar \cG _1))$ is zero,
and in this case its index in $K _0 (C^* (\cG))$ is the image of the Fredholm index of $[\tilde P]$.


\bigskip
\section{Renormalized index formulas for boundary groupoids of the form
$\cG = (M _0 \times M _0 ) \cup (\mathbb{R}^q \times M _1 \times M _1) ,$ with $q\geqslant 3$ odd}\label{OddRenorm}

In this section, we first consider boundary groupoids of the form
$$\cG = (M _0 \times M _0) \cup (G_1 \times M _1 \times M _1) ,$$
where $G_1$ is an exponential Lie group and $\dim G_1 >1 $ is odd.


We recall the $K$-theory computation in \cite{SoKtheory}.

\begin{theorem}\cite[Theorem 3.6]{SoKtheory}\label{OddCaseKTheory}
Let $\cG= (M_0\times M_0) \cup  G_1 \times M_1 \times M_1  \tto M $ be a boundary Lie groupoid with $G_1$
an exponential Lie group and $\dim G_1 >1 $ odd. Then
$$ K_0(C^*(\cG))\cong \mathbb{Z} , \quad K_1(C^*(\cG))\cong  \mathbb{Z}.$$
\end{theorem}

\begin{rem}\label{MainRem}
Putting the results of Theorem \ref{OddCaseKTheory} into \eqref{Comp4},
we get the exact sequence
\begin{equation}
\label{OddIsom}
0 \to K _1 (C^* (\cG)) \cong \bbZ \to K _1 (C ^* (G _1) \otimes \cK) \to K _0 (\cK) \cong \bbZ \to K _0 (C^* (\cG)) \cong \bbZ \to 0 ,
\end{equation}
which implies the inclusion from $K _0 (\cK) $ to $K _0 (C ^* (\cG))$ is an isomorphism.
The commutative diagram \eqref{Surj} now reads
\begin{equation}
\begin{CD}
K _1 (\fU / C^*(\cG _{M _0})) @>>> K _1 (\fU / C^* (\cG)) \\
@VV{\partial {}}V @VV{\partial {}}V  \\
K _0 (C^*(\cG _{M _0})) @>{\cong}>> K _0 (C^* (\cG))
\end{CD}.
\end{equation}
It follows that unlike the manifold with boundary case, for elliptic Fredholm operators, the $K$ theoretic index {\it is} the Fredholm index.
Moreover, the Fredholm index of such operator only depends on its principal symbol, as opposed to full symbol.
\end{rem}

From now on,  we confine our attention in computing an index formula for {\em renormalizable boundary grooupoids} considered in Example \ref{RenormEx},
which are boundary groupoids of the form
$$\cG = (M _0 \times M _0) \cup (\mathbb{R}^q \times M _1 \times M _1) ,$$
where $q\geqslant 3$ is an odd positive integer. We shall use $\cG$ to denote such a renormalizable boundary groupoid throughout this section.

Let $r := d (\cdot , M _1) $ be the Riemannian distance function on $M=M_0\cup M_1$ with respect to some Riemannian metric.
We assume that there are coordinate charts $(U _\alpha , \bx _\alpha )$ such that
$M_1 \subset  \cup _\alpha U _\alpha $, and on each $U _\alpha $ the image of the anchor map is spanned by
$$ r ^N\frac{\partial }{\partial x _1} , \cdots, r^N \frac{\partial }{\partial x _q},
\frac {\partial }{\partial y_1} , \cdots, \frac {\partial }{\partial y_{\dim M _1}} .$$
Note that these vector fields are smooth at $0$, which implies that $N $ is necessarily even.
In cylindrical coordinates the structural vector fields are spanned by
$$ r ^N \frac{\partial }{\partial r} , r ^{N -1} \frac{\partial}{\partial \theta _1}, \cdots, r ^{N -1} \frac{\partial}{\partial \theta _{q-1}},
\frac {\partial }{\partial y_1} , \cdots, \frac {\partial }{\partial y_{\dim M _1}}. $$

The Riemannian density of the induced metric on $M _0$,
which is defined by fixing a metric on $A \to M$, restricting to $M _0$,
and then identifying $A |_{M _0} $ with $T M_0$ via the inverse of the anchor map, is of the form
\begin{equation}
\label{Vol}
\mu _{M_0} := r ^{- q N} \mu _M \, (= r^{-N q + (q-1)}\,dr \,d\theta \, dy)
\end{equation}
for some smooth density $\mu _M$ on $M$ restricted to $M _0$.
Without loss of generality we shall assume $\mu _M$ is the Lebesgue measure with respect to each $\bx _\alpha$.

The volume form \eqref{Vol} implies $M_0$ is of polynomial volume growth.
By Lemma \ref{FunctCal}, the space of Schwartz kernels $S^* (\cG)$
and $\Psi ^0 (\cG , E) \oplus S ^* (\cG)$ are also closed under holomorphic functional calculus in $\fU (\cG)$.
In particular, $\cup _m \Psi ^m (\cG, E) \oplus S^* (\cG)$ satisfies all axioms in \cite[(i)-(vii),($\sigma$),($\psi$)]{Nistor;CplxPwr}.
It follows from \cite[Theorem 7.2]{Nistor;CplxPwr} that
for any uniformly supported elliptic, positive definite operator $Q$ of order 1
(for example $Q = (\varDelta + I) ^{\frac{1}{2}}$, see \cite{Nistor;GeomOp}),
the complex powers $Q ^z$ are well defined as pseudo-differential operators lying in
$\Psi ^{-z} (\cG , E) \oplus \cS ^* (\cG)$.

For any $P \in \Psi ^m (\cG,E)$, the operator
$P Q ^{- \tau} $ is a pseudo-differential operator of order $m - \tau$.
Observe that all $\bs$-fibers of $\cG$ are manifolds with bounded geometry and by \cite[A1 Theorem 3.7]{Shubin;BdGeom},
for any $k \in \bbN$, $P Q ^{- \tau} $ is a $C^k $ kernel whenever the real part of $\tau$ is sufficiently large.
Moreover, by \cite[Section 5]{Nistor;Hom1} and \cite[Lemma 1]{Nistor;Hom2}, $P Q ^{- \tau} $ is a holomorphic family for $\tau$ with sufficiently large real part.
Hence it makes sense to consider
\begin{eqnarray}
\label{RenormTr}
Z(P; \tau, z) &:=& \int_{M_0} \tr (r^{z}P Q^{-\tau} |_M \big) \, \mu_{M_0}, \\ \nonumber
&=& \int_{M _0} \tr (r ^{z - q N} \, P \,Q ^{- \tau} |_M ) \, \mu_M \quad \tau, z \in \bbC,
\end{eqnarray}
where we regard $r ^{z - q N}$ as the operator defined by multiplying by a scalar function.
Note that in \eqref{RenormTr} it suffices to restrict to the open dense subset $M _0 \subset M$,
hence after using cylindrical coordinates one is essentially back to the cusp calculus case,
and one can apply \cite[Section 5]{Nistor;Hom1} and \cite[Lemma 1]{Nistor;Hom2} to see that $Z (P; \tau, z)$ is a meromorphic family of $\tau $ and $z$,
with at most simple pole at $\tau = z = 0 $.
Hence we make the following definition.

\begin{dfn}
\label{RenormDfn}
Given a holomorphic family $B (\tau , z)$ of pseudo-differential operators, its {\em renormalized trace} is defined to be
\begin{eqnarray*}
\widehat{\Tr} (B) &:=& \frac{\partial ^2}{\partial \tau \partial z} \Big|_{\tau = 0, z = 0}
\int_{M _0} \tau z \tr \big(r^{z} B (\tau , z) Q ^{- \tau} \big |_{M } \big) \, \mu_{M_0} \\
&=& \frac{\partial ^2}{\partial \tau \partial z} \Big|_{\tau = 0, z = 0}
\int_{M _0} \tau z \tr \big(r ^{z - q N} B (\tau , z) Q ^{- \tau} \big |_{M } \big) \, \mu_M .
\end{eqnarray*}
\end{dfn}

In view of Remark \ref{MainRem}, to derive an index formula it suffices to consider the Fredholm index.
Suppose now that $D \in \rD ^1 (\cG , E)$ is a fully elliptic, first order differential operator.
Using the results of \cite[Section 4]{So;FullCal},
there exists a smooth pseudo-differential parametrix $\varPhi$ of $D$
up to finite rank operators and vanishing up to infinite order at $\cG _1$.
By the Calder\'on's formula, we obtain the following proposition.

\begin{proposition}
With the notation as above, the Fredholm index of $D$ equals
\begin{equation}\label{Cauldron}
 \partial([D])= \widehat\Tr ([D, \varPhi]) \, \Big(=\int_{M_0}  \tr ([D, \varPhi] |_M ) \, \mu_{M_0} \Big)  .
\end{equation}
\end{proposition}

Our goal is to figure out $\widehat\Tr([D, \varPhi])$, which will give us an index formula.
Following the arguments of \cite[Proposition 8]{Nistor;Hom2},
we have by direct calculation

\begin{equation}
\label{ResEq}
r ^z [D, \varPhi] Q ^{- \tau} = (r ^z D r ^{- z} - D) r ^z \varPhi Q ^{- \tau}
+ r ^z \varPhi (Q ^{- \tau} D Q ^\tau - D) Q ^{- \tau} + [D, r ^z \varPhi Q ^{- \tau}].
\end{equation}
The right-hand side of Equation \eqref{ResEq} has three terms. It is clear that the last term has vanishing trace
(because when the real part of $z$ and $\tau$ is sufficiently large the renormalized trace is just the usual trace),
hence it suffices to work out the first two terms.

Let us consider the first term of Equation \eqref{ResEq} which corresponds to the $\eta$ term in \cite[Equation (18)]{Nistor;Hom2}.
However, unlike the manifold with boundary case, we now show the renormalized trace vanishes in our case.
\begin{lem}
The renormalized trace
\begin{equation}
\label{OddEta}
\frac{\partial ^2}{\partial \tau \partial z} \Big|_{\tau = 0, z = 0}
\tau z \int _{M _0} \tr ( r ^{- q N} (r ^z D r ^{- z} - D) r ^z \varPhi Q ^{- \tau} |_M)\, \mu _M  = 0 .
\end{equation}
\end{lem}
\begin{proof}
The proof reduces to a local calculation by partition of unity argument.
We parameterize each $U _\alpha$ using cylindrical coordinates and let $\sigma : U _\alpha \to U _\alpha $
(shrinking $U _\alpha$ if necessary) be the cylindrical anti-pole map, i.e.,
$$ (r , \sigma (\theta ), y ) := (- (r, \theta), y) .$$

By assumption, the operator $D$ is the composition of one of the vector fields in \eqref{StructuralVF} and tensors.
Taking local trivialization, one writes $D$ in cylindrical coordinates:
\begin{align*}
D = & D _0 (r, \theta, y) r ^N \frac{\partial }{\partial r} +
D _1 (r, \theta, y) r ^{N -1} \frac{\partial}{\partial \theta _1} + \cdots
+ D _{q-1} (r, \theta, y) r ^{N -1} \frac{\partial}{\partial \theta _{q-1}} \\
& + D' _1 (r, \theta, y) \frac{\partial }{\partial y_1} + \cdots + D' _{\dim M _1} (r, \theta, y) \frac{\partial }{\partial y_{\dim M _1}} + D''(r, \theta, y).
\end{align*}
Consider $(r ^z D r ^{- z} - D) r ^z$. One has
$$ \big[ r ^N \frac{\partial}{\partial r} , r ^z \big] = z r ^{z + N - 1},$$
which implies
$$ \big(r ^z (r ^N \frac{\partial}{\partial r}) r ^{- z} - r ^N \frac{\partial}{\partial r} \big) r ^z \varPhi Q ^{- \tau}
= z r ^{z + N -1} \varPhi Q ^{- \tau};$$
while other commutators vanish.
It follows that
$$ (r ^z D r ^{- z} - D) r ^z \varPhi Q ^{- \tau} = z r ^{z-1} r ^N D _0 \varPhi Q ^{- \tau} .$$
In cylindrical coordinates, denoting the volume measure of the $q-1$ dimensional sphere by $d \theta$, we have
\begin{align*}
& \tau z \,\int _{M _0} \tr ( r ^{- q N} (r ^z D r ^{- z} - D) r ^z \varPhi Q ^{- \tau} |_M)\, \mu _M\\
&=\tau z^2 \int_{M_0} r^{z-1+N-qN}\, \tr(D_0 \varPhi  Q^{-\tau} |_M)\,\mu_M\\
&=\tau z^2 \int_{M_0} r^{z-1+N-qN}\, \tr(D_0 \varPhi  Q^{-\tau} |_M)\,r^{q-1} \, dr\,d\theta\,dy \\
&= \tau z \int_{M_0} z\, r^{z-1+N} \,r^{-qN+q-1}\, \tr(D_0 \varPhi  Q^{-\tau} |_M) \, dr\,d\theta\,dy.
\end{align*}
Hence the left-hand side in \eqref{OddEta} becomes (by linearity)
\begin{equation}\label{LHS}
\int z\, r ^{z-1 + N} r ^{- q N + q - 1} \tr (D _0 \varPhi Q ^{- \tau} |_M) \,d r \, d \theta \, d y .
\end{equation}

Expand $D _0 $ as a (finite) Taylor series
$$ D _0 = A _0 (\theta, y) + r A _1 (\theta, y) + \cdots ,$$
for some matrices $A _j (\theta, y)$.
Since $D _0 r ^N \frac{\partial }{\partial r}$ is a smooth differential operator at $\{ r = 0 \}$, it follows that
$$ A _j (\sigma (\theta), y) = - (-1) ^j A _j (\theta , y), \quad j=0, 1, 2, \cdots .$$
Similarly, for any $\tau$ with sufficiently large real part,
we have noted in the paragraph before Definition \ref{RenormDfn} that $\varPhi Q ^{- \tau}$ is given by a $C^k$ kernel of sufficiently large $k$.
Observe furthermore that $\{ r = 0 \}$ is high co-dimensional sub-manifold.
It follows that $\tr (D _0 \varPhi Q ^{- \tau} |_M)$ has (finite) Taylor expansion with respect to $r$:
$$\tr (D _0 \varPhi Q ^{- \tau} |_M) = F _0 (\theta , y) + F _1 (\theta , y) r + \cdots .$$
satisfying
\begin{equation}\label{Parity1}
F _j (\sigma (\theta), y) = -(-1) ^j F _j (\theta , y), \quad j = 0, 1, 2, \cdots,
\end{equation}
because the derivatives are continuous across $\{ r = 0 \}$.

In \eqref{LHS}, one replaces the trace factor $\tr (D _0 \varPhi Q ^{- \tau} |_M)$ by its Taylor expansion, integrates with respect to $r$,
and takes the constant term to get:
\begin{equation}
\label{PwrInt}
\int F _{(N-1)(q-1)} (\theta , y) d \theta d y.
\end{equation}
Since $(q - 1)(N - 1) $ is even, it follows from parity \eqref{Parity1} that the $\theta$ integral in \eqref{PwrInt} also vanishes
for all $\tau$ with sufficiently positive real part.
Hence the lemma follows by meromorphic extension.
\end{proof}

It remains to consider the second term of Equation \eqref{ResEq}.
This term has been computed in \cite[Proposition 12]{Nistor;Hom2}, and we shall briefly recall the result here.
\begin{lem}
One has
$$  \widehat\Tr (\varPhi (Q ^{- \tau} D Q ^\tau - D) Q ^{- \tau}) = \int _{M _0} r ^{z - q N} \, a_0 \, \mu_M \Big| _{z = 0}
=\Big( \int _{M _0} r ^{z} \, a_0 \, \mu_{M_0} \Big| _{z = 0}  \Big),$$
where $a _0 $ is the constant term in the $t \to 0 $ asymptotic expansion of the super-trace of heat kernel \cite[Equation (26)]{Nistor;Hom2}
$$ \sTr \Big( \exp \big(- t
\big( \begin{smallmatrix} 0 & D^* \\ D & 0 \end{smallmatrix} \big) \big) \Big)
= \Tr \big( \exp (-t D^* D) - \exp (- t D D^*) \big).$$
\end{lem}
\begin{proof}
We fix
\begin{align*}
Q := Q _1 := & (D D ^* + \varPi _{\ker (D^*)}) ^{\frac{1}{2}} \\
Q _2 := & (D ^* D + \varPi _{\ker (D)}) ^{\frac{1}{2}} \\
\varPhi := & D ^* Q _1 ^{-2} = Q _2 ^{-2} D ^*,
\end{align*}
where $\varPi _{\ker (D)}$ and $\varPi _{\ker (D^*)} $ denote orthogonal projections (in $L ^2 (M _0)$) onto $\ker (D)$ and $\ker (D^*)$, respectively.
Note that for any $\tau \in \bbC$,
$$ Q _1 ^ {-\tau } D = D Q _2 ^{- \tau }.$$
It follows that for any $\tau, z \in \bbC$:
\begin{align*}
\varPhi (Q ^{- \tau} D Q ^\tau - D) Q ^{- \tau}
= & \: \varPhi (Q _1 ^{- \tau } D - D Q _1 ^{- \tau}) \\
= & \: \varPhi D (Q _2 ^{- \tau } - Q _1 ^{- \tau}) \\
 \widehat\Tr (\varPhi (Q ^{- \tau} D Q ^\tau - D) Q ^{- \tau})
= & \widehat\Tr (Q _2 ^{- \tau } - Q _1 ^{- \tau}).
\end{align*}
Since $\tau ^{-1} (Q _2 ^{- \tau } - Q _1 ^{- \tau})$ is a holomorphic family,
by \cite[Lemma 1]{Nistor;Hom2}, the integral
$$\int _{M _0} \tr (r ^{z - q N} (Q _2 ^{- \tau } - Q _1 ^{- \tau})) \, \mu_M$$
is holomorphic at $\tau = 0$.
Using the Mellin transform one gets for any $\tau $ with sufficiently large real part
$$ \tr (Q _2 ^{- \tau } - Q _1 ^{- \tau} |_M)
= \frac{1}{\Gamma (\tau / 2)} \int _0 ^\infty t ^{\frac{\tau}{2} - 1}
\tr \big( e ^{- t (D D ^* + \varPi _{\ker (D^*)}) } - e ^{- t (D ^* D + \varPi _{\ker (D^*)}) } |_M \big) d t .$$
Since it is clear that
\begin{align*}
e ^{- t (D D ^* + \varPi _{\ker (D^*)}) } = & e ^{- t D D ^* } + (e ^t - 1) \varPi _{\ker (D^*) } \\
e ^{- t (D ^* D + \varPi _{\ker (D^*)}) } = & e ^{- t D ^* D} + (e ^t - 1) \varPi _{\ker (D^*) },
\end{align*}
the claim follows.
\end{proof}

By the discussion in Section \ref{Relation}, for any elliptic operator $D$ in $\rD ^1 (\cG , E)$, we can find a fully elliptic operator $\widetilde{D}$ in $\rD ^1 (\cG , E)$
such that $[D]=[\widetilde{D}]$, i.e., in the same $K$-theory class.
Summarizing the calculations above, we obtain theorem \ref{Main}.

\begin{theorem}
\label{RenormIndex}
For any elliptic, differential operator $D \in \rD ^1 (\cG , E)$,
the index of $D$ in $K _0 (C^* (\cG))$ equals
\begin{equation}
\partial ([D (D D^* + \id )^{- \frac{1}{2}} ])
= \int_{M _0} r ^{z - q N} \, a _0 \, \mu_M \Big |_{z = 0},
\end{equation}
where $a _0 $ is the constant term in the $t \to 0 $ asymptotic expansion of the super-trace of heat kernel.
\end{theorem}
%

Recall that Bohlen and Schrohe in \cite{Bohlen;HeatIndex} proved that for a Dirac type operator $D$ on a Lie manifold
(and by \cite{Bohlen;Dirac} general fully elliptic operators reduce to the Dirac case), its index is
$$\partial ([D (D D^* + \id )^{-\frac {1}{2}}]) = \int _M \AS (D) + \eta (D),$$
where
$$ \eta (D) = \frac{1}{2} \int _0 ^\infty {}^\nu \Tr _s ([ D , e ^{- t D ^2}]) d t $$
and the ${}^\nu \Tr _s$ is the Hadamard (i.e. cutoff) regularized super-trace.

Comparing with Equation \eqref{OddEta}, we obtain the following non-trivial result.

\begin{corollary}
With the notation as above, we have
$$ \int _0 ^\infty {}^\nu \Tr _s ([ D , e ^{- t D ^2}]) d t = 0 .$$
\end{corollary}

\bigskip

\section{The case that $q = 1$}\label{qoneCase}

We turn to consider boundary groupoids with one-dimensional isotropy subgroups.

\subsection{Index formula for connected manifolds with connected boundaries}\label{BoundaryCase}
In this subsection, we assume that $M = M _0 \cup M _1$ is a connected manifold with connected boundary,
where we denote by $M _0$ and $M _1$ the interior and boundary of $M$, respectively.
Then the groupoid that we consider in this subsection is of the form
$$\cG = (M _0 \times M _0 ) \cup (\bbR \times M _1 \times M _1),$$
where $M_1$ is the connected boundary of $M$.

With a modified argument in \cite[Remark 3.7]{SoKtheory}, we conclude the following $K$-theory result.

\begin{lem}
Under the same assumption as above, one has
$$ K _0 (C^* (\cG)) \cong \{ 0 \} , \quad K _1 (C ^* (\cG)) \cong \{ 0 \} ,$$
that is, the $K$-theoretic index is trivial.
\end{lem}


\begin{rem}
We have the exact sequence \cite{Nistor;BoundaryK}
$$ 0 \to K _1 (C^* (\cG)) \to \bbZ \xrightarrow{\partial} \bbZ \to K _0 (C^* (\cG)) .$$
It follows that the connecting map $\partial$ is an isomorphism.
More concretely, we have
\begin{equation}
\label{Incidence}
\partial (1) = 1 .
\end{equation}
\end{rem}

Recall that $M$ here is a manifold with embedded boundary \cite{Melrose;Book},
therefore one has a boundary defining function $r \in C^\infty (M)$
(which satisfies $M _1 = r ^{-1} (0) $ and $d r \neq 0 $ on $M_1$).
The structural vector fields near $M _1$ are spanned by
$$ r ^N \frac{\partial }{\partial r}, \frac{\partial}{\partial y_1}, \cdots , \frac{\partial}{\partial y _{\dim M _1}}, $$
for some integer $N$,
then formulas computing the Fredholm index of an elliptic operator $D$ that is invertible on $\cG _1$ is well known ---
one of which is by adapting the arguments in Section \ref{OddRenorm}. We just briefly recall the result here.

\begin{proposition}\label{APS}
\cite[Proposition 12, Proposition 13]{Nistor;Hom2}
For any first order elliptic differential operator $D$ on the groupoid
$\cG = (M _0 \times M _0) \cup (\bbR \times M _1 \times M _1)$ defined as above,
that is invertible on $\bbR \times M _1 \times M _1$,
the Fredholm index of $D (D D^* + \id )^{- \frac{1}{2}}$ equals
\begin{equation}
\partial ([D (D D^* + \id )^{- \frac{1}{2}} ])
= \int _{M _0} a _0  \,\mu_M + \eta (D),
\end{equation}
where $a _0 $ is the constant term in $t \to 0 $ asymptotic expansion of the trace of heat kernel
$$\tr \big( (e ^{-t D D ^*} - e ^{-t D^* D}) \big|_M \big),$$
and
$$\eta (D) := \frac{\partial ^2}{\partial \tau \partial z} \Big|_{\tau = 0, z = 0}
\int _{M _0} \tr \big(r ^{z - N} ( D - r ^{-z} D r ^z) ( D D^* + \id )^{-\frac{1}{2} - \tau} \big) \mu _M $$
depends only on $D |_{\bbR \times M _1 \times M _1}$.
\end{proposition}

\medskip
\subsection{An Index formula for closed manifolds separated by interior hypersurfaces}\label{TwoHyperSurfaces}

In this subsection, let $M $ be a compact manifold without boundary, $M _1$ a closed sub-manifold of co-dimension one.
We suppose further that
$$M = M' _0 \cup M _1 \cup M'' _0, $$
where $M' _0$ and $M'' _0 $ are open and connected submanifolds, and their closures $\bar M' _0$ and $\bar M'' _0 $ are manifolds with the same boundary $M_1$.
In particular, our assumption implies the normal bundle of $M _1$ is trivial.

We consider the subspace of vector fields tangential to $M _1$.
The corresponding Lie groupoid is
$$ \cG = (M' _0 \times M' _0) \cup (M'' _0 \times M'' _0) \cup (\bbR \times M _1 \times M _1),$$
which contains
$$ \cG' := (M' _0 \times M' _0) \cup (\bbR \times M _1 \times M _1),
\quad \cG'' = (M'' _0 \times M'' _0) \cup (\bbR \times M _1 \times M _1)$$
as invariant subspaces.

\begin{example}
Let $\bbR$ act on itself by $(x,t) \rightarrow x\cdot e^t$. There are three orbits $(-\infty, 0)$, $\{0\}$, and $(0,\infty)$.
Then set $M'=(-\infty,0)$, $M''=(0,\infty)$, and $M_1=\{0\}$.
The action groupoid $\cG = \bbR \rtimes \bbR \tto \bbR$ satisfies the above conditions.
\end{example}

We first compute $K$-theory of the groupoid $C^*$-algebra of $\cG$ defined as above.

\begin{theorem}
Under the same assumption as above, one has
$$K _0 ( C^* (\cG ) ) \cong \bbZ. $$
\end{theorem}
\begin{proof}
For the composition series of $\cG$, we have
$$ C^* ((M' _0 \times M' _0) \cup (M'' _0 \times M'' _0))
= C^* (M' _0 \times M' _0) \oplus C^* (M'' _0 \times M'' _0)
= \cK \oplus \cK ,$$
where each copy of $\cK$ corresponds to one component.
Therefore Diagram \eqref{Comp4} becomes the exact sequence:
\begin{equation}
\label{CD2.2}
K _1 (C^* (\bbR \times M _1 \times M _1 )) \cong \bbZ \xrightarrow{\partial {}} \bbZ \oplus \bbZ
\xrightarrow{\iota _*} K _0 (C^* (\cG)) \to 0.
\end{equation}
We compute the connecting map $\partial$ using the morphism of exact sequences
$$
\begin{CD}
0 @>>>
\begin{array}{c}
C^* (M' _0 \times M' _0) \\
\oplus C^* (M'' _0 \times M'' _0)
\end{array}
@>>> C^* (\cG) @>>> C^* (\cG ) / (\cK \oplus \cK) @>>> 0 \\
@. @| @VVV @VVV @. \\
0 @>>>
\cK \oplus \cK @>>>
\begin{array}{c}
C^* (\cG' ) \\
\oplus C^* (\cG'')
\end{array}
@>>>
\begin{array}{c}
C^* (\cG' ) / \cK \\
\oplus C^* (\cG'' ) / \cK
\end{array}
@>>> 0,
\end{CD}
$$
where the middle column map is defined by restriction to $\cG'$ and $\cG'$.
Thus, one gets the commutative diagram:
$$
\begin{CD}
K _1 (C^* (\bbR \times M _1 \times M _1 ))
@>{\partial {}}>> K _0 (C^* (M' _0 \times M' _0)) \oplus K _0 (C^* (M'' _0 \times M'' _0)) \\
@VVV @| \\
K _1 (C^* (\cG' ) / \cK) \oplus K _ 1 (C^* (\cG'' ) / \cK) @>{\partial \oplus \partial {}}>> \bbZ \oplus \bbZ.
\end{CD}
$$
Here, the bottom row is just two copies of the connection map in \eqref{Incidence}.
On the other hand, we have
$$ C^* (\bbR \times M _1 \times M _1 ) = C^* (\cG' ) / C^* (M' _0 \times M' _0) = C^* (\cG'' ) / C^* (M'' _0 \times M'' _0), $$
and the left column is just two copies of the identity map.
It follows that
$$\partial (1) = 1 \oplus 1 .$$
Hence $K _0 ( C^* (\cG ) ) \cong \bbZ $ and $\iota _* :  \bbZ \oplus \bbZ \to \bbK _0 (C ^* (\cG)) $ in \eqref{CD2.2} is realized as
\begin{equation}
\label{AntiInd}
\iota _* (1 \oplus 0 ) = 1, \iota _* (0 \oplus 1) = -1 . \qedhere
\end{equation}
\end{proof}

We then turn to exploit the index of an elliptic pseudo-differential operators on $\cG$.

\smallskip
First, given any elliptic pseudo-differential operator $P$ that is invertible on $\cG _1$,
we consider the index of $P$ in $ K _0 ( C^* ((M' _0 \times M' _0) \cup (M'' _0 \times M'' _0))) $.
One has the homomorphism of short exact sequences:
$$
\begin{CD}
0 \to
\begin{array}{c}
C^* ((M' _0 \times M' _0) \\
\quad \cup (M'' _0 \times M'' _0))
\end{array}
@>>> \fU (\cG)
@>>> \fU (\cG) \Big/
\begin{array}{c}
C^* ((M' _0 \times M' _0) \\
\quad \cup (M'' _0 \times M'' _0))
\end{array}
\to 0 \\
@| @VVV @VVV \\
0 \to
\begin{array}{c}
C^* (M' _0 \times M' _0)\\
\oplus C ^* (M'' _0 \times M'' _0)
\end{array}
@>>>
\begin{array}{c}
\fU (\cG') \\
\oplus \fU (\cG'')
\end{array}
@>>>
\begin{array}{c}
\fU (\cG') / C^* (M' _0 \times M' _0) \\
\oplus \fU (\cG'') / C^* (M'' _0 \times M'' _0)
\end{array}
\to 0,
\end{CD}
$$
where the middle map is given by (the extension of) restricting the pseudo-differential operator on $\cG$
to the invariant subspaces $\cG'$ and $\cG''$,
and it is clear that the restriction map induces a map for the right column.

By the naturality of the index map, one has
\begin{equation}
\label{CD3}
\begin{CD}
K _1 \Big(\fU (\cG) \Big/
\begin{array}{c}
C^* ((M' _0 \times M' _0) \\
\quad \cup (M'' _0 \times M'' _0))
\end{array}
\Big)
@>{\partial}>>
K _0 \Big(
\begin{array}{c}
C^* ((M' _0 \times M' _0) \\
\quad \cup (M'' _0 \times M'' _0))
\end{array}
\Big)  \\
@VVV @| \\
\begin{array}{c}
K _1 (\fU (\cG') / C ^*(M' _0 \times M' _0)) \\
\oplus K _1 (\fU (\cG'') / C^* (M'' _0 \times M'' _0))
\end{array}
@>{\partial \oplus \partial}>> K (\cK \oplus \cK) \cong \bbZ \oplus \bbZ,
\end{CD}
\end{equation}
where ${\partial} \oplus {\partial}$ is just the Fredholm index maps for each copy.

To conclude, the index of $P$ in $K _0 ( C^* ((M' _0 \times M' _0) \cup (M'' _0 \times M'' _0))) $
is the direct sum of the Fredholm indices of $P |_{\cG '}$ and $P|_{\cG ''}$,
and the Fredholm index of $P $ is clearly the arithmetic sum of the two indices in $\bbZ$.
\smallskip

We are in position to derive an index formula (in $K_0 (C^* \cG)$) for first order elliptic differential operators on $\cG$.

\begin{theorem}\label{OneInd}
Let $D$ be a first order elliptic differential operator in $\rD ^1 (\cG, E) + S ^* (\cG)$. The following index formula holds
$$ \partial ([D (D D)^* + \id )^{- \frac{1}{2}} ])
= \int _{M' _0} a' _0 - \int _{M'' _0} a'' _0 \in K _0 (C^* (\cG)) \cong \bbZ,$$
where $a_0'$ and $a_0''$ are defined in the same way as in Proposition \ref{APS} .
\end{theorem}

\begin{proof}
We first suppose that $D$ is a first order (pseudo-)elliptic differential operator that is invertible on $\bbR \times M _1 \times M _1$,
and compare its index in $K_0(C^*(\cG))$ and $K_0 (C^* ((M' _0 \times M' _0) \cup (M'' _0 \times M'' _0)))$
using Equation \eqref{Surj}. One gets
\begin{equation}
\begin{CD}
K _1 (\fU / C^* ((M' _0 \times M' _0) \cup (M'' _0 \times M'' _0))) @>>> K _0 (\cK \oplus \cK) = \bbZ \oplus \bbZ \\
@VVV @VV{\iota _*}V \\
K _1 (\fU / C ^* (\cG)) @>{\partial {}}>> K _0 ( C ^* (\cG)),
\end{CD}
\end{equation}
where the right column is just \eqref{AntiInd}. It follows from Proposition \ref{APS} that
\begin{equation}
\partial ( [D (D D)^* + \id )^{- \frac{1}{2}} ])
= \int _{M' _0} a' _0 - \int _{M'' _0} a'' _0 \in K _0 (\cC ^* (\cG)) \cong \bbZ,
\end{equation}
where $a' _0$ and $a''_0$ are respectively the constant term in $t \to 0 $ asymptotic expansion of the traces
of the heat kernels
$$ e ^{- t D D ^* }|_{M' _0 \times M' _0 } -  e ^{- t D ^* D }|_{M' _0 \times M' _0 }\,
\text{ and }\,
e ^{- t D D ^*} |_{M'' _0 \times M'' _0 } - e ^{- t D ^* D} |_{M'' _0 \times M'' _0 } .$$

Next, let us consider general case when $D$ belongs to $\rD ^1 (\cG, E) + S^* (\cG)$ and is elliptic.
By the discussion in Section \ref{Relation}, there exists a Fredholm operator $\widetilde D $ of the form
$$ \widetilde D = D + R,\,\, R \in S ^* (\cG)$$
such that
$$ \partial ([\widetilde D (\widetilde D \widetilde D^* + \id )^{- \frac{1}{2}} ])
= \partial ([D (D D^* + \id )^{- \frac{1}{2}} ]) \in K _0 (C^* (\cG)).$$
Since $\widetilde D$ and $D$ have the same principal symbol, the
asymptotic expansions as $t \to 0$ of the traces of their heat kernels have the same constant.
So we are done.
\end{proof}

\begin{remark}
The above theorem shows that in this case the $K$-theory index is the difference of Fredholm indices.
\end{remark}

\medskip
\section{Concluding remarks}\label{summary}
In this paper we have derived index formulas for some boundary groupoid $C^*$-algebras.
It is obvious that we are not considering the most general cases.
It seems clear that more examples can be computed by combining the above results with combinatorial techniques.

It seems that for the groupoids constructed in Example \ref{RenormEx},
one can define classifying space as in \cite{SoKtheory} (with $M = \bbS^N \times \bbS ^q$).
Then an embedding $M \to \bbS^N \times \bbS ^q$ induces a groupoid homomorphism and the arguments of \cite{Nistor;TopIndCorn} can be used to derive an index formula.
That should work for all $q \geq 2$.

However, constructing such groupoid homomorphisms may not be feasible for general boundary groupoids with $r=1$.
When $M_1$ is a point and $G_1 = \bbR ^q$, one can adapt the argument index formula \cite{Nistor;Perturbed} for `asymptotically Abelian manifolds,
and then try to generalize this approach by using excision.

It would also be interesting to consider manifolds with fibered boundary.
In other words, one considers groupoids which are a union of fibered products:
$$ \cG = \bigcup _i G _i \times (M _i \times _{B _i }M _i).$$
This includes the examples considered in \cite{Debord;FiberedCorners,Mazzeo;EdgeRev} and many others.
It should be particularly interesting to generalize the Dirac operator and index pairing approach
\cite{Rouse;BlupInd,Bohlen;Dirac,Debord;ConnesThom,Yamashita;FiberedBdK} to this case.

Also it is not clear how a renormalized integral can be defined for boundary groupoids with $r \geq 2$.
Last but not the least the $\eta$-term of the even dimensional case is still mysterious.
We expect some interesting new invariants arising from the relevant index formulas.

\def\cprime{$'$} \def\cprime{$'$}


\begin{thebibliography}{10}

\bibitem{Rouse;BlupInd}
I.~Akrour and P.~Carrillo~Rouse.
\newblock Longitudinal b-operators, blups and index theorems.
\newblock arXiv:1711.11197v3, 2018.

\bibitem{Albin;EdgeInd}
P.~Albin.
\newblock A renormalized index theorem for some complete asymptotically regular
  metrics: the {G}auss-{B}onnet theorem.
\newblock {\em Adv. Math.}, 213(1):1--52, 2007.

\bibitem{Nistor;Riem}
B.~Ammann, R.~Lauter, and V.~Nistor.
\newblock On the geometry of {R}iemannian manifolds with a {L}ie structure at
  infinity.
\newblock {\em Int. J. Math. Math. Sci.}, (1-4):161--193, 2004.

\bibitem{Nistor;LieMfld}
B.~Ammann, R.~Lauter, and V.~Nistor.
\newblock Pseudodifferential operators on manifolds with a {L}ie structure at
  infinity.
\newblock {\em Ann. of Math. (2)}, 165(3):717--747, 2007.

\bibitem{Nistor;CplxPwr}
B.~Ammann, R.~Lauter, V.~Nistor, and A.~Vasy.
\newblock Complex powers and non-compact manifolds.
\newblock {\em Comm. Partial Differential Equations}, 29(5-6):671--705, 2004.

\bibitem{Skandalis;SingFoliation1}
I.~Androulidakis and G.~Skandalis.
\newblock The analytic index of elliptic pseudodifferential operators on a
  singular foliation.
\newblock {\em J. K-Theory}, 8(3):363--385, 2011.

\bibitem{Skandalis;SingFoliation2}
I.~Androulidakis and G.~Skandalis.
\newblock Pseudodifferential calculus on a singular foliation.
\newblock {\em J. Noncommut. Geom.}, 5(1):125--152, 2011.

\bibitem{Bohlen;Rev1}
K.~Bohlen.
\newblock Groupoids and singular manifolds.
\newblock arXiv:1601.04166v2, 2017.


\bibitem{Bohlen;HeatIndex}
K.~Bohlen and E.~Schrohe.
\newblock Getzler rescaling via adiabatic deformation and a renormalized index
  formula.
\newblock {\em J. Math. Pures Appl. (9)}, 120:220--252, 2018.

\bibitem{Bohlen;Dirac}
K.~Bohlen and J.-M. Lescure.
\newblock A geometric approach of {$K$-homology} for {Lie} manifolds.
\newblock arXiv:1904.04069v2, 2019.

\bibitem{Rouse;CornerFredholm}
P.~Carrillo~Rouse and J.-M. Lescure.
\newblock Geometric obstructions for {F}redholm boundary conditions for
  manifolds with corners.
\newblock {\em Ann. K-Theory}, 3(3):523--563, 2018.

\bibitem{Monthubert;BdCoh}
P.~Carrillo~Rouse, J.-M. Lescure, and B.~Monthubert.
\newblock A cohomological formula for the {A}tiyah-{P}atodi-{S}inger index on
  manifolds with boundary.
\newblock {\em J. Topol. Anal.}, 6(1):27--74, 2014.

\bibitem{Monthubert;BoundaryInd}
P.~Carrillo-Rouse and B.~Monthubert.
\newblock An index theorem for manifolds with boundary.
\newblock {\em C. R. Math. Acad. Sci. Paris}, 347(23-24):1393--1398, 2009.

\bibitem{SoKtheory}
P.~Carrillo~Rouse and B.~K. So.
\newblock K-theory and index theory for some boundary groupoids.
\newblock {\em Results Math.}, 75(4):Paper No. 172, 20pp, 2020.

\bibitem{Nistor;Perturbed}
C.~Carvalho and V.~Nistor.
\newblock An index formula for perturbed {D}irac operators on {L}ie manifolds.
\newblock {\em J. Geom. Anal.}, 24(4):1808--1843, 2014.

\bibitem{Nistor;Fredholm1}
C.~Carvalho, V.~Nistor, and Y.~Qiao.
\newblock Fredholm criteria for pseudodifferential operators and induced
  representations of groupoid algebras.
\newblock {\em Electron. Res. Announc. Math. Sci.}, 24:68--77, 2017.

\bibitem{Nistor;Fredholm2}
C.~Carvalho, V.~Nistor, and Y.~Qiao.
\newblock Fredholm conditions on non-compact manifolds: theory and examples.
\newblock In {\em Operator theory, operator algebras, and matrix theory},
  volume 267 of {\em Oper. Theory Adv. Appl.}, pages 79--122.
  Birkh\"{a}user/Springer, Cham, 2018.

\bibitem{Come19}
R.~C\^{o}me.
\newblock The {F}redholm property for groupoids is a local property.
\newblock {\em Results Math.}, 74(4):Paper No. 160, 33pp, 2019.

\bibitem{Connes;ThomIsomorphism}
A.~Connes.
\newblock An analogue of the {T}hom isomorphism for crossed products of a
  {$C^{\ast} $}-algebra by an action of {${\bf R}$}.
\newblock {\em Adv. in Math.}, 39(1):31--55, 1981.

\bibitem{Connes;Book}
A.~Connes.
\newblock {\em Noncommutative geometry}.
\newblock Academic Press, Inc., San Diego, CA, 1994.

\bibitem{Fern'd;IntAlgebroid}
M.~Crainic and R.~L. Fernandes.
\newblock Integrability of {L}ie brackets.
\newblock {\em Ann. of Math. (2)}, 157(2):575--620, 2003.

\bibitem{Debord;IntAlgebroid}
C.~Debord.
\newblock Holonomy groupoids of singular foliations.
\newblock {\em J. Differential Geom.}, 58(3):467--500, 2001.

\bibitem{Nistor;ConicalInd}
C.~Debord, J.-M. Lescure, and V.~Nistor.
\newblock Groupoids and an index theorem for conical pseudo-manifolds.
\newblock {\em J. Reine Angew. Math.}, 628:1--35, 2009.

\bibitem{Debord;FiberedCorners}
C.~Debord, J.-M. Lescure, and F.~Rochon.
\newblock Pseudodifferential operators on manifolds with fibred corners.
\newblock {\em Ann. Inst. Fourier (Grenoble)}, 65(4):1799--1880, 2015.


\bibitem{Debord;ConnesThom}
C.~Debord and G.~Skandalis.
\newblock Lie groupoids, exact sequences, {C}onnes-{T}hom elements, connecting
  maps and index maps.
\newblock {\em J. Geom. Phys.}, 129:255--268, 2018.

\bibitem{Debord;Blup}
C.~Debord and G.~Skandalis.
\newblock Blowup constructions for {Lie} groupoids and a {Boutet de Monvel}
  type calculus.
\newblock {\em M\"{u}nster J. Math.}, 14(1):1--40, 2021.

\bibitem{Nistor;Family}
R.~Lauter, B.~Monthubert, and V.~Nistor.
\newblock Pseudodifferential analysis on continuous family groupoids.
\newblock {\em Doc. Math.}, 5:625--655, 2000.

\bibitem{Nistor;Funct}
R.~Lauter, B.~Monthubert, and V.~Nistor.
\newblock Spectral invariance for certain algebras of pseudodifferential
  operators.
\newblock {\em J. Inst. Math. Jussieu}, 4(3):405--442, 2005.

\bibitem{Nistor;GeomOp}
R.~Lauter and V.~Nistor.
\newblock Analysis of geometric operators on open manifolds: a groupoid
  approach.
\newblock In {\em Quantization of singular symplectic quotients}, volume 198 of
  {\em Progr. Math.}, pages 181--229. Birkh\"{a}user, Basel, 2001.

\bibitem{Monthubert;BdK}
P.-Y. Le~Gall and B.~Monthubert.
\newblock {$K$}-theory of the indicial algebra of a manifold with corners.
\newblock {\em $K$-Theory}, 23(2):105--113, 2001.

\bibitem{Loya;HeatIndex}
P.~Loya.
\newblock The index of {$b$}-pseudodifferential operators on manifolds with
  corners.
\newblock {\em Ann. Global Anal. Geom.}, 27(2):101--133, 2005.

\bibitem{Lu;PoissonCohNotes}
J.-H. Lu.
\newblock A note on {P}oisson homogeneous spaces.
\newblock In {\em Poisson geometry in mathematics and physics}, volume 450 of
  {\em Contemp. Math.}, pages 173--198. Amer. Math. Soc., Providence, RI, 2008.

\bibitem{MacBook87}
K.~Mackenzie.
\newblock {\em Lie groupoids and {L}ie algebroids in differential geometry},
  volume 124 of {\em London Mathematical Society Lecture Note Series}.
\newblock Cambridge University Press, Cambridge, 1987.

\bibitem{MacBook05}
K.~Mackenzie.
\newblock {\em General theory of {L}ie groupoids and {L}ie algebroids}, volume
  213 of {\em London Mathematical Society Lecture Note Series}.
\newblock Cambridge University Press, Cambridge, 2005.

\bibitem{Mazzeo;EdgeRev}
R.~Mazzeo.
\newblock Elliptic theory of differential edge operators. {I}.
\newblock {\em Comm. Partial Differential Equations}, 16(10):1615--1664, 1991.

\bibitem{Melrose;Book}
R.~Melrose.
\newblock {\em The {A}tiyah-{P}atodi-{S}inger index theorem}, volume~4 of {\em
  Research Notes in Mathematics}.
\newblock A K Peters, Ltd., Wellesley, MA, 1993.

\bibitem{Nistor;Hom1}
R.~Melrose and V.~Nistor.
\newblock Homology of pseudodifferential operators I. Manifolds with boundary.
\newblock arXiv:9606005v2

\bibitem{Nistor;BoundaryK}
R.~Melrose and V.~Nistor.
\newblock {$K$}-theory of {$C^*$}-algebras of {$b$}-pseudodifferential
  operators.
\newblock {\em Geom. Funct. Anal.}, 8(1):88--122, 1998.

\bibitem{M'bert;CornerGroupoids}
B~Monthubert.
\newblock Pseudodifferential calculus on manifolds with corners and groupoids.
\newblock {\em Proc. Amer. Math. Soc.}, 127(10):2871--2881, 1999.

\bibitem{Nistor;TopIndCorn}
B.~Monthubert and V.~Nistor.
\newblock A topological index theorem for manifolds with corners.
\newblock {\em Compos. Math.}, 148(2):640--668, 2012.

\bibitem{MonPie}
B.~Monthubert and F.~Pierrot.
\newblock Indice analytique et groupo\"{\i}des de {L}ie.
\newblock {\em C. R. Acad. Sci. Paris S\'{e}r. I Math.}, 325(2):193--198, 1997.

\bibitem{Nistor;Hom2}
S.~Moroianu and V.~Nistor.
\newblock Index and homology of pseudodifferential operators on manifolds with
  boundary.
\newblock In {\em Perspectives in operator algebras and mathematical physics},
  volume~8 of {\em Theta Ser. Adv. Math.}, pages 123--148. Theta, Bucharest,
  2008.

\bibitem{Nistor;IntAlg'oid}
V.~Nistor.
\newblock Groupoids and the integration of {L}ie algebroids.
\newblock {\em J. Math. Soc. Japan}, 52(4):847--868, 2000.

\bibitem{Nistor;Polyhedral3}
V.~Nistor.
\newblock Pseudodifferential operators on non-compact manifolds and analysis on
  polyhedral domains.
\newblock In {\em Spectral geometry of manifolds with boundary and
  decomposition of manifolds}, volume 366 of {\em Contemp. Math.}, pages
  307--328. Amer. Math. Soc., Providence, RI, 2005.

\bibitem{NWX;GroupoidPdO}
V.~Nistor, A.~Weinstein, and P.~Xu.
\newblock Pseudodifferential operators on differential groupoids.
\newblock {\em Pacific J. Math.}, 189(1):117--152, 1999.

\bibitem{RenBook80}
J.~Renault.
\newblock {\em A groupoid approach to {$C^{\ast} $}-algebras}, volume 793 of
  {\em Lecture Notes in Mathematics}.
\newblock Springer, Berlin, 1980.

\bibitem{Shubin;BdGeom}
M.~A. Shubin.
\newblock {\em Methodes semi-classiques Vol 1, Spectra of elliptic operators on non-compact manifolds},
volume 207 of {\em Asterisque}.
\newblock 1992

\bibitem{So;PhD}
B.~K. So.
\newblock Pseudo-differential operators, heat calculus and index theory of
  groupoids satisfying the {Lauter-Nistor} condition.
\newblock PhD thesis, The University of Warwick, 2010.

\bibitem{So;FullCal}
B.K. So.
\newblock On the full calculus of pseudo-differential operators on boundary
  groupoids with polynomial growth.
\newblock {\em Adv. Math.}, 237:1--32, 2013.

\bibitem{Yamashita;FiberedBdK}
M.~Yamashita.
\newblock A topological approach to indices of geometric operators on manifolds
  with fibered boundaries.
\newblock {\em Comm. Math. Phys.}, 377(1):77--147, 2020.
\end{thebibliography}
\end{document}